\documentclass[12pt]{amsart}
\usepackage{amsmath,amssymb,amsthm,amscd}
\numberwithin{equation}{section}

\newtheorem{theo}{Theorem}[section]
\newtheorem{lemm}{Lemma}[section]
\newtheorem{coro}{Corollary}[section]
\newtheorem{rema}{Remark}[section]

\def\begeq{\begin{equation}}
\def\endeq{\end{equation}}

\begin{document}

\title{A nonexistence result on harmonic diffeomorphisms between punctured spaces}
\author{Shi-Zhong Du$^\dagger $ \&   Xu-Qian Fan$^*$}
\thanks{$^\dagger$ Research partially supported by the National
Natural Science Foundation of China (11101106).
}
\address{The School of Natural Sciences and Humanities,
            Shenzhen Graduate School, The Harbin Institute of Technology, Shenzhen, 518055, P. R. China.}
\email{szdu@hitsz.edu.cn}

\thanks{$^*$ Research partially supported by the National
Natural Science Foundation of China (11291240139).
}
\address{Department of
Mathematics, Jinan University, Guangzhou,
 510632,
P. R. China.}
\email{txqfan@jnu.edu.cn}

\renewcommand{\subjclassname}{%
  \textup{2000} Mathematics Subject Classification}
\date{July 2014}

\begin{abstract}
In this paper, we will prove a result of nonexistence on harmonic diffeomorphisms between punctured spaces. In particular, we will given an elementary proof to the nonexistence of rotationally symmetric harmonic diffeomorphisms from the punctured Euclidean space onto the punctured hyperbolic space.
\end{abstract}
\maketitle\markboth{Shi-Zhong Du $\&$ Xu-Qian Fan}{A non-exitence result on harmonic diffeomorphisms}

\section{Introduction}
We will study the problem of the existence about harmonic diffeomorphism between punctured spaces, in particular, from the punctured $\mathbb{R}^n$ to the punctured $\mathbb{H}^n$. This is a natural generalization of the question mentioned by  Schoen \cite{rs}, which is about
the existence, or nonexistence, of a harmonic diffeomorphism from $\mathbb{C}$ onto $\mathbb{D}$  with the
Poincar\'{e} metric. Up to present, many beautiful results for the asymptotic behaviors of harmonic maps from $\mathbb{C}$ to $\mathbb{D}$  were obtained, see for example \cite{wt,httw,atw,aw}, or the survey \cite{wt2} by Wan and the references therein. On the other hand, Collin and Rosenberg \cite{cr} constructed harmonic diffeomorphisms in 2010. In addition, there are also many papers focusing on the rotationally symmetric case, for example \cite{ta,ta2,rr,cl}.  One of these results is the nonexistence of the special harmonic diffeomorphisms from $\mathbb{C}$ onto $\mathbb{D}$, even for general dimension. Recently, some works have studied the case of annular topological type, for example \cite{cdf,lh,cdf2}. While some results are related to the Nitsche's type inequalities \cite{n,hz,iko,kd4}.

Conversely, Heinz \cite{hz} obtained the nonexistence result for these mappings from $\mathbb{D}$ onto the flat plane $\mathbb{C}$  in 1952.

For more results about the question of the existence of harmonic diffeomorphisms between Riemannian manifolds, see for example \cite{ak}-\cite{w}.

In this paper, we will generalize the result for the nonexistence of rotationally symmetric harmonic diffeomorphism from $\mathbb{C}^*$  onto $\mathbb{D}^*$ to general dimension $n\geq 2$.  Actually, we will study the problem for more general setting. More precisely, let us denote
\begin{equation*}
\begin{split}
(\mathbb{R}^n,E)&=(\mathbb{S}^{n-1}\times [0,\infty),r^2d\theta^2+dr^2),\\
(M^n,G)&=(\mathbb{S}^{n-1}\times [0,\infty),(g(r))^2 d\theta^2+dr^2),
\end{split}
\end{equation*}
where $C^2$ function $g(r)\geq 0$ and $(\mathbb{S}^{n-1},d\theta^2)$ is the $(n-1)$-dimensional sphere, and
$$\mathbb{R}_*^n=\mathbb{R}^n\setminus \{0\} \textrm{ and }  M_*^n=M^n\setminus \{0\}.$$
 We will prove the following result  in the next section.
\begin{theo}\label{thm1}
For $n\geq 2$, there is no rotationally symmetric harmonic diffeomorphism from $(\mathbb{R}_*^n, E)$ onto $(M_*^n,G)$ provided that the Riemannian metric $G$  satisfies the following conditions:

 (1) $g(0)g'(0)\geq0$,

 (2) $(g(r)g'(r))'\geq 0$ for all $r\geq 0$,

 (3) $\sup_{r>0}[r^{-1}g(r)g'(r)]>\frac{(n-2)^2}{n-1}$.
\end{theo}

 Clearly, $\mathbb{H}^n=(\mathbb{S}^{n-1}\times [0,\infty),\sinh^2r d\theta^2+dr^2)$, and $\sinh r$ satisfies  conditions (1)-(3). Let
$$\mathbb{H}_*^n=\mathbb{H}^n\setminus \{0\},$$
we have the following corollary.
\begin{coro}\label{cor1}
For $n\geq 2$, there is no rotationally symmetric harmonic diffeomorphism from $(\mathbb{R}_*^n, E)$ onto $\mathbb{H}_*^n$ with the hyperbolic metric.
\end{coro}
From this corollary, one can get the well-known result: There is no rotationally symmetric harmonic diffeomorphism from Euclidean space $\mathbb{R}^n$ onto hyperbolic space $\mathbb{H}^n$.
\section*{Acknowledgments}
The authors would like to thank Li Chen for his useful discussion. The author(XQ) would like to thank Prof. Luen-fai Tam for his worthy advice.

\section{Proof of Theorem \ref{thm1}}
For convenience, let us first recall the definition about the harmonic maps from $\mathbb{R}^n$ to $M^n$. Let $(x^1,\cdots,x^n)$ and $(u^1,\cdots,u^n)$ be the standard coordinates on $\mathbb{R}^n$ and $M^n$ respectively, and $\Gamma_{ij}^k$ be the Christoffel symbol of the Levi-Civita connection on $M^n$.  A $C^2$ map $u$ from $\mathbb{R}^n$ to $M^n$ is harmonic if and only if $u$ satisfies
\begin{equation} \label{eqdef}
\triangle u^i+\sum_{\alpha=1}^n\Gamma_{jk}^i(u)\frac{\partial u^j}{\partial x^\alpha}\frac{\partial u^k}{\partial x^\alpha}=0\quad \textrm{ for } i=1,\cdots, n.
\end{equation}
For more information about harmonic maps, see for example \cite{sy}.

\begin{proof}[Proof of Theorem \ref{thm1}]
 We want to prove this theorem by contradiction. Suppose $u$ is a rotationally symmetric harmonic diffeomorphism from $\mathbb{R}_*^n$ onto $M_*^n,$ then we can assume $u(r,\theta)=(y(r),\theta)$, and by (2.3) in \cite{ta} (or (1.2) in \cite{cl}) and \eqref{eqdef}, $y(r)$ should satisfy the following ODE.
  \begin{equation}\label{e2}
     y''+(n-1)r^{-1}y'-(n-1)r^{-2}g(y)g'(y)=0.
  \end{equation}
  In addition, $u$ is a diffeomorphism, so $y'$ cannot be zero. So $y$ should satisfy one of the following  conditions \eqref{bfc1} and \eqref{bfc2}, where
 \begin{equation}\label{bfc1}
     y(0)=0, y(r)\to +\infty \textrm{ as } r\to+\infty, \textrm{ and } y'>0 \textrm{ for } 0<r<\infty,
  \end{equation}
or
\begin{equation}\label{bfc2}
      y(r)\to +\infty \textrm{ as } r\to 0+, \lim_{r\to+\infty}y(r)=0, \textrm{ and } y'<0 \textrm{ for } 0<r<\infty.
  \end{equation}

By the monotonicity of $y$ in $r$, we can regard $r$ as function of $y$ with the properites
  $$
   y'=\frac{1}{r'}, \ \ \ y''=-\frac{r''}{r'^3}.
  $$
Consequently, \eqref{e2} is changed to be
  \begin{equation}\label{e3}
     -\frac{r''}{r'^3}+\frac{n-1}{rr'}-(n-1)\frac{1}{r^2}g(y)g'(y)=0,
  \end{equation}
or
  \begin{equation}\label{e4}
   -\frac{r''}{r}+(n-1)\Big(\frac{r'}{r}\Big)^2-(n-1)\Big(\frac{r'}{r}\Big)^3g(y)g'(y)=0.
  \end{equation}
Setting $x=\ln r$, we have
  $$
    x'=\frac{r'}{r}, \ \ \
    x''=\frac{r''}{r}-\Big(\frac{r'}{r}\Big)^2.
  $$
From \eqref{e4}, we can get
  \begin{equation}\label{e5}
      x''-(n-2)x'^2+(n-1)x'^3g(y)g'(y)=0.
  \end{equation}
Let $z=x'=\frac{r'(y)}{r(y)}$, we can see that \eqref{e5} is equivalent to
  \begin{equation}\label{e6}
     z'-(n-2)z^2-(n-1)z^3g(y)g'(y)=0.
  \end{equation}

Clearly, this is an Abel's equation for $n\geq 3$. It suffices for us to show that the
non-existence of solution to \eqref{e6} with condition \eqref{bfc1} or condition \eqref{bfc2}.

Now let us prove the first part, that is, equation \eqref{e6} with condition \eqref{bfc1} has no solution. In fact, this part is well-known, see for example \cite{rr,cl}. But for completeness, we will give an alternative proof here.
Suppose such a solution exists. Since $z>0$ for $y\in (0,\infty)$, using \eqref{e6} and the fact: $g(y)g'(y)\geq 0$ for $y\geq 0$ by conditions (1) (2), we can get
   \begin{eqnarray*}
     z'&=&(n-2)z^2+(n-1)z^3g(y)g'(y)\\
       &\geq& (n-2)z^2.
   \end{eqnarray*}
So
   $$
     (z^{-1})'\leq -(n-2).
   $$
Integrating over $y$, we get
   $$
     z^{-1}(y)-z^{-1}(1)\leq -(n-2)(y-1)
   $$
   for $y\geq 1.$
Letting $y\to+\infty$, we have
   $$
     z^{-1}(y)<0.
   $$
This contradicts the property $z>0$ for $y\in (0,\infty)$. Hence we finish the proof of this part.

We still need to prove the second part, that is, equation \eqref{e6} with condition \eqref{bfc2} has no solution.
Now let us use conditions (2) (3) to get the following result first.
\begin{lemm}\label{l1}
  Let $z(y)$ be a solution to \eqref{e6} with condition \eqref{bfc2}, we have
     $$
       (n-2)+(n-1)z(y)g(y)g'(y)\geq 0
     $$
for $y\in(0,+\infty)$.
\end{lemm}

\noindent\textbf{Proof of Lemma \ref{l1}.} Let us assume on the contrary, that is,
   $$
    (n-2)+(n-1)z(y_0)g(y_0)g'(y_0)<0
   $$
for some $y_0\in(0,+\infty)$. Setting
   $$
     \Sigma\equiv\{\omega\in(y_0,+\infty):\ \ \
     (n-2)+(n-1)z(y)g(y)g'(y)<0 \mbox{ holds in }
     (y_0,\omega)\}.
   $$
It's clearly that $\Sigma$ is a closed set in $(y_0,+\infty)$, and is nonempty by continuity of $z,\ g$ and $g'$.
If we can prove that $\Sigma$ is also open in $(y_0,+\infty)$, then
   $$
    \Sigma\equiv (y_0,+\infty)
   $$
by connection of $(y_0,+\infty)$. In fact, by equation \eqref{e6} and the definition of $\Sigma$, we can get $z(y)$ is monotone decreasing in
$(y_0,\omega_0)$ for any
$\omega_0\in\Sigma$, which implies that
   $$
    (n-2)+(n-1)z(y)g(y)g'(y)
   $$
is also decreasing in $(y_0,\omega_0)$ since $z(y)<0$ for $y\in(0,\infty)$. So
\begin{equation*}
\begin{split}
    &(n-2)+(n-1)z(\omega_0)g(\omega_0)g'(\omega_0)\\
    &<(n-2)+(n-1)z(y_0)g(y_0)g'(y_0)<-\delta
    \end{split}
\end{equation*}
for some positive constant $\delta$. That's to say, $\omega_0$ is an interior point of $\Sigma$
by continuity. As a result, $\Sigma$ is open. Hence
   $$
     \Sigma=(y_0,+\infty)
   $$
and for all $y>y_0$,
\begin{equation*}
\begin{split}
    &(n-2)+(n-1)z(y)g(y)g'(y)\\
    &<(n-2)+(n-1)z(y_0)g(y_0)g'(y_0)<-\delta.
    \end{split}
\end{equation*}

Consequently, we can get
  \begin{eqnarray*}
    z'&=& z^2[(n-2)+(n-1)zg(y)g'(y)]\\
      &\leq& -\delta z^2
  \end{eqnarray*}
  for $y\geq y_0.$
So
   $$
    (z^{-1})'>\delta.
   $$
Hence
   $$
    z^{-1}(y)>z^{-1}(y_0)+\delta(y-y_0)>0
   $$
provided that $y>y_0$ is large enough, which contradicts the property $z<0$ for $y\in(0,\infty)$.
Therefore, Lemma \ref{l1} holds. \hfill{} $\Box$\\

Combining  \eqref{e6}, the result of Lemma \ref{l1} and condition \eqref{bfc2}, we can get the following results.
\begin{coro}\label{coro1}
 Let $z(y)$ as the same as in Lemma \ref{l1}, we have
\begin{equation}\label{eqzp1}
    z'(y)\geq 0 \mbox{ for } y>0,
    \end{equation}
     then
\begin{equation}\label{eqzp2}
    \lim_{y\to 0^+}z(y)=-\infty.
    \end{equation}
\end{coro}
\noindent\textbf{Proof of Corollary \ref{coro1}.}
\eqref{eqzp1} is a direct result of  \eqref{e6} and the result of Lemma \ref{l1}.

By condition \eqref{bfc2}, we have $\lim_{y\to 0^+}r(y)=+\infty$, so $\lim_{y\to 0^+}\ln r(y)=+\infty$, then $z=(\ln r)'$ is unbounded on $(0,\epsilon)$ for any $\epsilon>0$. On the other hand, from \eqref{eqzp1}, $z=(\ln r)'$ is an increasing function in $y$ for $y\geq 0$. Hence for  $y\to 0^+$, $z=(\ln r)'\to -\infty.$ \hfill{} $\Box$

\vspace{0.2cm}
Let us continue to prove the second part.
Setting $Z(y)=z^{-1}(y)$ and using equation \eqref{e6} again, we
have
   \begin{equation}\label{e7}
     -Z'=(n-2)+(n-1)\frac{g(y)g'(y)}{Z}.
   \end{equation}
Letting $w=Z^2(y)$, we get
   \begin{equation}\label{e8}
     -\frac{1}{2}w'=-(n-2)\sqrt{w}+(n-1)g(y)g'(y).
   \end{equation}
So
   \begin{equation}\label{e9}
     \begin{cases}
        w'=2(n-2)\sqrt{w}-2(n-1)g(y)g'(y)\\
         \lim_{y\to0^+}w(y)=0
     \end{cases}
   \end{equation}
where $\lim_{y\to0^+}w(y)=0$ comes from \eqref{eqzp2}. By
\eqref{e9} and nonnegativity of $g(y)g'(y)$, we have
   \begin{eqnarray*}
     w'&\leq& 2(n-2)\sqrt{w}.
   \end{eqnarray*}
Consequently, we can get
   $$
    \sqrt{w(y)}\leq (n-2)y
   $$
   for $y>0$.
Substituting this into \eqref{e9}, we can get
 \begin{eqnarray*}
 \begin{split}
     0&\leq
     w'(y)=2(n-2)\sqrt{w}-2(n-1)g(y)g'(y)\\
     &\leq2(n-2)^2y-2(n-1)g(y)g'(y)
     \end{split}
   \end{eqnarray*}
for all $y>0$. From this, we can get
        $$
         \sup_{y>0}[y^{-1}g(y)g'(y)]\leq\frac{(n-2)^2}{n-1}.
        $$
 This contradicts condition (3). Hence equation \eqref{e6} with condition \eqref{bfc2} has no solution.

Combining the results of these two parts, we can conclude that no such harmonic diffeomorphism
exists from $(\mathbb{R}_*^n, E)$ onto $(M_*^n, G)$ provided that $G$ satisfies conditions (1)-(3). So we finish the proof of Theorem \ref{thm1}.
\end{proof}

\begin{rema}\label{rema1} The``diffeomorphism" in this statement of Theorem \ref{thm1} can be replaced by ``homeomorphism". That is, for $n\geq 2$, there is no rotationally symmetric harmonic homeomorphism from $(\mathbb{R}_*^n, E)$  onto $(M_*^n,G)$ provided that the Riemannian metric $G$ satisfies conditions (1)(3) and $g(y)g'(y)>0$ for $y>0$.
\end{rema}
\noindent\textbf{Proof of Remark \ref{rema1}.}
We need to show that such a rotationally symmetric harmonic homeomorphism is a diffeomorphism. This result should exist somewhere for more general setting, but we didn't find a suitable reference. For readable, we explain the reason for our special case. Noting that if such a rotationally symmetric harmonic homeomorphism exists, then the corresponding $C^2$ function $y$ is positive with $y'\geq 0$ for all $r\in (0,\infty)$, or with $y'\leq 0$ for all $r\in (0,\infty)$. Clearly, it suffices for us to prove $y'\not=0$ for any $r\in (0,\infty)$.

We will first show that if $y>0$ and $y'\geq 0$ for all $r\in (0,\infty)$, then $y'>0$ for all $r\in (0,\infty)$. Suppose $y'(r_0)=0$ for some $r_0>0$, then $y''(r_0)=0$ by maximum principle. Substituting these into \eqref{e2}, we can get a contradiction.

Similarly, we can get that if $y>0$ and  $y'\leq 0$ for all $r\in (0,\infty)$, then $y'(r)< 0$  for all  $r\in (0,\infty).$

Combining the results of the previous two cases, we can conclude  that  $y'\not=0$ for any $r\in (0,\infty)$, so such a homeomorphism is diffeomorphic. Hence by Theorem  \ref{thm1}, we can get the result of Remark \ref{rema1}.
 \hfill{} $\Box$

\end{document}